\documentclass[12pt]{article}
\title{Amenability of linear-activity automaton groups}

\date{April 2009}

\author{\thanks{University of Toronto}
\and Gideon Amir\footnotemark[1] \thanks{Bar-Ilan University, email: gidi.amir@gmail.com}
  \and Omer Angel\footnotemark[1] \thanks{University of British Columbia,
    Supported in part by NSERC of Canada. email: angel@math.ubc.ca }
  \and B\'alint Vir\'ag\footnotemark[1] \thanks{Research supported in part
      by NSF grant \#DMS-0206781. email: balint@math.utoronto.ca}}

\usepackage{graphicx,amsmath,amsthm,amssymb,hyperref}
\usepackage[round,authoryear]{natbib}

\theoremstyle{plain}
    \newtheorem{theorem}{Theorem}[section]
    \newtheorem{lemma}[theorem]{Lemma}
    \newtheorem{proposition}[theorem]{Proposition}
    \newtheorem{corollary}[theorem]{Corollary}
    \newtheorem{claim}[theorem]{Claim}

    \newtheorem{maintheorem}{Theorem}
    \newtheorem{conjecture}[maintheorem]{Conjecture}
    \newtheorem{mainprop}[maintheorem]{Proposition}

\theoremstyle{definition} 
    \newtheorem{definition}[theorem]{Definition}

    \newtheorem*{example*}{Example}
    \newtheorem*{question*}{Question}

\newcommand{\eps}{\varepsilon}

\newcommand{\Eps}{\mathcal E}
\newcommand{\Z}{{\mathbb Z}}
\newcommand{\N}{{\mathbb N}}
\newcommand{\E}{{\mathbb E}}
\renewcommand{\P}{{\mathbb P}}

\newcommand{\R}{{\mathbb R}}

\newcommand{\M}{\mathfrak M}
\renewcommand{\L}{\mathfrak L}
\newcommand{\tree }{\mathbb T}
\newcommand{\pT}{\mathcal T}
\newcommand{\cQ}{\mathcal Q}
\newcommand{\cR}{\mathcal R}
\newcommand{\cA}{\mathcal A}
\newcommand{\id}{{\tt 1}}
\newcommand{\Aut}{\operatorname{Aut}}
\newcommand{\Sym}{\operatorname{Sym}}
\newcommand{\stab}{\operatorname{stab}}
\newcommand{\llangle}{\langle\hspace{-0.2em}\langle}
\newcommand{\rrangle}{\rangle\hspace{-0.2em}\rangle}
\newcommand{\ang}[1]{{\llangle #1 \rrangle}}
\newcommand{\Ent}{H}
\newcommand{\Entlim}{\Ent_\infty}
\newcommand{\oneone}{{\biggl(\begin{matrix}1\\1\end{matrix}\biggr)}}

\newcommand{\lemref}[1]{Lemma~\ref{L:#1}}
\renewcommand{\bar}{\overline}

\DeclareMathOperator{\supp}{supp}

\begin{document}

\maketitle

\begin{abstract}
  We prove that every linear-activity automaton group is amenable.
  The proof is based on showing that a random walk
  on a specially constructed degree $1$ automaton group --- the mother
  group --- has asymptotic entropy $0$. Our result answers an open
  question by Nekrashevych in the Kourovka notebook, and
  gives a partial answer to a question of Sidki.
\end{abstract}

\begin{figure}[h]
\begin{center}
\includegraphics[width=3.5in]{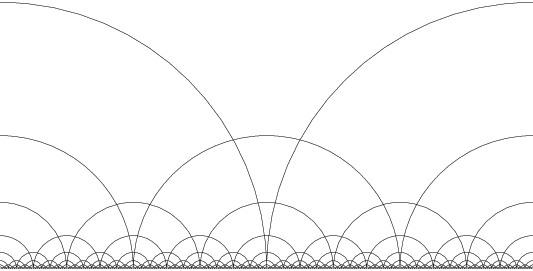}
\caption{\label{f:longrange} A Schreier graph for a linear-activity group}
\end{center}
\end{figure}

\section{Introduction}

Automaton groups are the core of the algebraic theory of
fractals. Just as fractals do in geometry, automata groups
form a rich new world within group theory. This world has
been a source of many interesting examples, but until
recently it resisted a general theory. Yet the simplicity
of the definitions and the richness of examples suggest the
existence of such a theory. The goal of this paper is to
make a step in this direction by proving that all
linear-activity groups are amenable.

Automaton groups arise in various areas of mathematics;
examples include:

\begin{description}
\item{\it The Grigorchuk group}, \citep[see][]{Grigorchuk84} which has
  faster than polynomial but subexponential growth, answering the question
  of \citet{Milnor68} about the existence of such groups.

\item{\it The Basilica group}, a finitely generated amenable
    group containing a non-cyclic free semi-group. Its level
    Schreier graphs have a limit which is homeomorphic to the
    Basilica fractal, that is the Julia set of the polynomial
    $z^2-1$, see \citet{gz02a}, Figure 17 and Theorem 9.7 in
    \citet*{BGN}, and  \citet{BV05}.

\item{\it The Hanoi towers group}, the group of possible moves in the Hanoi
  Towers game on three pegs, a game introduced by \`Edouard Lucas in 1883.
  Its level Schreier graphs are discrete Sierpinski gaskets
  \citep[see][]{GS06}.

\item{\it The long-range group}, an interesting group whose Schreier graphs
  (Figure \ref{f:longrange}) were studied by \citet{bh05} in the context of
  long-range percolation theory.

\item{\it The lamplighter group on
$\mathbb Z$}, the first example of a group with a Cayley
graph that has discrete spectrum, providing also a
counterexample to the strong Atiyah conjecture. See
\citet{gz01}.
\end{description}

Automaton groups are subgroups of the automorphism group
$\Aut(\tree _m)$ of the rooted infinite $m$-ary tree for some $m$. Every
element of $\Aut(\tree _m)$ can be decomposed as
\[
g=\ang{ g_0,\ldots ,g_{m-1}}\sigma, \qquad g_i\in \Aut(\tree _m), \ \
\sigma \in \Sym(m),
\]
where the $g_i$, called {\bf first-level sections}, now act
on the subtrees rooted at the children of the root of
$\tree _m$, and $\sigma$ permutes these subtrees. An {\bf
automaton} $A$ is a finite subset of $\Aut(\tree _m)$ so
that first level sections of elements of $A$ are also in
$A$. An {\bf automaton group} is a group generated by an
automaton. See, e.g., \citet{SG07} and references there for a general
survey on automaton groups.

A systematic study of automaton groups was initiated by
 \citet{Sidki00}, who introduced the concept of {\bf
activity growth} (or more briefly, activity), a measure of complexity
for automaton groups. We defer the precise definition for later,
but note that the activity can be either polynomial (of any degree) or
exponential. There are exponential activity growth automaton
groups that are isomorphic to the free group
\citep[see][]{GlasnerMozes05, Vorobets07}. However, one expects
polynomial activity automaton groups to be smaller. In particular, in
contrast with most examples of finitely generated non-amenable
groups, \citet{Sidki00} showed that polynomial activity automaton groups
have no free subgroups (the works of \cite{O80},
\cite{Ad82}, and \cite{OS02}, showed such groups exist (finitely generated,
non-abelian with no free subgroup),
but the examples were quite hard to come by). This prompted
\citet{Sidki04} to ask the following natural question.

\begin{question*}
  Are all polynomial activity automaton groups amenable?
\end{question*}

It seems that the structure of polynomial automaton groups
depends a lot on the degree, and we do not have a
conjectured answer to Sidki's question. The first result in
this direction is due to \citet*{BKN}, who showed that all
degree 0 (also called bounded) activity automaton groups are
amenable. Their proof uses a variant of the self-similar
random walk idea introduced in \citet{BV05} to prove the
amenability of the Basilica group and later streamlined and
generalized in \citet{Kaimanovich05}. The goal of this
paper is to show the following.

\begin{maintheorem}\label{T:main}
  All linear-activity automaton groups are amenable.
\end{maintheorem}

\begin{example*}\label{e:long range}
  A simple example of our theorem is given by the {\bf long-range group}:
  Let $b$ act on the integers by increasing them by 1, and let $a$ act on
  integers by increasing them by the lowest power of $2$ by which they are not
  divisible (with $a(0)=0$). The group generated by $a,b$ is called the
  long-range
  group; its Schreier graph is shown in Figure~\ref{f:longrange}. Its
  automaton has two states defined recursively by
  \[
  a=\ang{a,b}, \qquad b=\ang{b,1}(01).
  \]
\end{example*}

\defcitealias{Kourovka}{Kourovka notebook (2006)}
\defcitealias{Guido}{Guido's book of conjectures (2008)}
The question of the amenability of the long-range group was
posed by Nekrashevych, in the \citetalias{Kourovka},
Question 16.74, and also in \citetalias{Guido}, Conjecture
35.9.

Our paper builds on previous work in the theory, including
\citet{BV05}, \citet{Kaimanovich05}, and \citet{BKN}. The
first part of the proof is to construct a family of
linear-activity automaton groups
--- the {\bf mother groups}--- which are then shown to
contain subgroups isomorphic to every linear-activity
automaton group. This has been done in \citet{BKN} for
bounded automaton groups. We show how to perform this step
for polynomial automata of any degree. It then suffices to
show that the linear-activity mother groups are amenable.

The second part is to find a random walk on the mother group with
0 asymptotic entropy. We can only do this for linear-activity
mother groups. In fact, we conjecture that there is a phase
transition for this question, namely such walks only exist up to
degree $2$. Recall that a random walk is symmetric if its step
distribution $\mu$ satisfies $\mu(g)=\mu(g^{-1})$ for all group
elements $g$.

\begin{conjecture}
  The mother group of degree $d$ has a symmetric random walk whose step
  distribution is supported on a finite generating set and whose asymptotic
  entropy is zero if and only if $d=0,1,2$.
\end{conjecture}

For $d=1$ this follows from this paper. The high degree ($d\ge3$) case is
proved\footnote{The conjecture is actually proved for all mother groups of
  degree $3$ or more other than the degree $3$ mother group over the binary
  alphabet. We believe this to be an artifact of the proof.} in
\cite{AV11}. For $d=2$ this conjecture is still open.

Compared to previous work the random walks we consider are no
longer self-similar. Despite starting from a finitely supported walk on the group, the proof involves the analysis of random walks with infinite support that arise through the evolution of the random walk step measure. The role previously played by self-similar walks is played by the evolution of random walk step measures that we can control. Such control requires a high level of regularity for the random walk. However, we believe that this is an artifact of the proof, and that the entropy bounds should hold for more general random walks.

\begin{conjecture}
  Every symmetric random walk with boundedly supported step distribution on
  a polynomial automaton group with degree $d\leq 2$ has zero asymptotic
  entropy.
\end{conjecture}

This conjecture is open even for $d=0$.

The main tool for the analysis of asymptotic entropy is study of the
so-called {\bf ascension operator}. Consider a random walk $X_n$ with step
distribution $\mu$ on an automaton group,
and let $v$ be the first child of the root of $\tree _m$.
The section of $X_n$ at $v$, looked at the times at
which the walk fixes $v$ forms another random walk with
step distribution $\mu'$. The map $T:\mu \mapsto \mu'$ is
called the ascension operator (see further details in
Section \ref{s:ascension}).

Ascension of random walks first considered in \cite{BV05}.
The ascension operator in this form appeared in
\citet{Kaimanovich05}, where the asymptotic entropy inequality
\mbox{$\Entlim(\mu) \le \Entlim(T\mu)$} was also proved. Iterating
this inequality, one gets $\Entlim(\mu) \le
\Entlim(T^n\mu)$. In this paper, we will analyze $T^n \mu$
in the case when these measures are not finitely supported
and not computable exactly.
The following proposition allows us to relate
$\Entlim(\mu)$ to the asymptotic entropy of limit points of
the sequence of measures $T^n\mu$.

\begin{mainprop}[Upper semi-continuity of asymptotic entropy]
  \label{usc H} If $\nu_n,\nu$ are probability measures on a countable
  group so that $\nu_n\to \nu$ weakly and $H(\nu_n) \to H(\nu)$, then
  \[
  \limsup_{n\to\infty}(\Entlim(\nu_n))\le \Entlim(\nu).
  \]
\end{mainprop}

In light of Proposition~\ref{usc H}, it suffices to find a
subsequence along which $T^n\mu  \to \nu$ and $H (T^n\mu
)\to H(\nu)$, and $\nu$ has zero asymptotic entropy.
Perhaps surprisingly, it is not too difficult to show that
for appropriate $\mu$, any subsequential limit point $\nu$
has zero asymptotic entropy. For the other two claims, it
suffices to show the tightness and entropy-tightness
(defined below) of the sequence $T^n \mu$. Proving these
facts takes up a large part of this paper. Showing the
tightness of the sequence $T^n \mu$ is the main obstacle to
extending our proof to the degree $2$ case. For
convenient reference, we summarize the preceding discussion
in a theorem.

\begin{maintheorem}[Asymptotic entropy of automaton groups]
  \label{T:maintool}
  If the group generated by the support of $\mu$ acts transitively on all
  levels of the tree and the sequence $T^k \mu$ is entropy-tight, then for
  every subsequential weak limit point $\nu$ we have
  \[
  \Entlim(\mu) \le \Entlim(\nu).
  \]
\end{maintheorem}

The main challenge is to construct measures for which the
ascension operator is tractable. The measures that we consider are
based on the uniform measures $\bar q_i$ on certain finite
subgroups of the mother group $\mathfrak M$. They have the
property that if a probability measure $\mu $ on $\mathfrak M$ is
a convex combination of convolution products of the $\bar q_i$'s,
then so is $T\mu$ (i.e.\ the algebra generated by the $\bar q_i$ is
invariant to $T$). Thus the iterated ascensions $T^n \mu$ of such
measures can be understood in terms of the $\bar q_i$'s, which
have extra symmetry and can be controlled. Further details can be
found in Section~\ref{s:patterns}.

\paragraph{Organization.}
The structure of this paper is as follows. Sections
\ref{s:automata}--\ref{s:ascension} contain definitions, setup, background
review and proof of
some preliminary results; readers familiar with previous work on the
subject may want to skip these sections. Section \ref{s:automata} reviews
basic concepts related to automata groups. In Section \ref{s:entropy} we do
the same for entropy, and prove Proposition~\ref{usc H}. In
Section~\ref{s:ascension} we introduce and study the ascension operator.
Mother groups are defined in Section~\ref{s:mother}, and it is shown that
they contain all polynomial automata groups. In Section \ref{s:patterns} we
introduce a special class of measures, called patterns, and an algebraic
way to study them. We also define an ascension operator for patterns. In
Section \ref{s:general} we study properties of iterated ascension on
patterns. Sections~\ref{s:bounded} contains some preliminary results on
entropy of pattern measures, and finally, the main theorem is proved in Section
\ref{s:linear}.

\section{Automata and their groups}\label{s:automata}

\paragraph{Basic definitions.} Finite automata are the simplest interesting
model of computing; we first connect our definition of automata to the more
traditional one.

The space of words in alphabet $\{0,\dots,m-1\}$ has a natural tree
structure, with $\{wx\}_{x<m}$ being the children of the finite word $w$,
and the empty word $\emptyset$ being the root. Let $\tree_m$ denote this
tree. A {\bf finite automaton} on $m$ symbols is a finite set of states $A$
together with a map $A \to A^{m} \times \Sym(m)$ sending $a \mapsto
(a_0,\ldots, a_{m-1}, \sigma_a)$. We will use the notation
\[
a = \llangle  a_0,\ldots, a_{m-1} \rrangle  \sigma_a.
\]

An automaton acts on words in alphabet $\{0,\ldots,{m-1}\}$
sequentially. When the automaton is in a state $a$ and
reads a letter $x$, it outputs $x.\sigma_a$ and moves to
state $a_x$. From this state the automaton acts on the rest
of the word. Symbolically, for a word $xw$ (starting with a
letter $x$) we have the recursive definition
\begin{equation}\label{eq:group_action}
  (xw).a = (x.\sigma_a)(w.a_x).
\end{equation}
The first $k$ symbols of the output are determined by the first $k$ symbols
read, and the action is invertible. Note that the action is defined for for
both finite and infinite
words, and that the action on infinite words determines the action on
finite words and vice versa. It follows that each element $a\in A$ is an
automorphism of $\tree _m$. The {\bf automaton group} corresponding to an
automaton $A$ is the subgroup of $\Aut(\tree_m)$ generated by $A$.

The action \eqref{eq:group_action} corresponds to the
following {\bf multiplication rule}:
\[
\ang{a_0,\dots,a_d} \sigma \ang{b_0,\dots,b_d} \tau =
\ang{a_0b_{0.\sigma},\ldots,a_{m-1}b_{(m-1).\sigma}} \sigma
\tau.
\]
This multiplication rule can be used to define automaton
groups without any reference to automorphisms of the tree.
However, keeping the action on the tree in mind makes some
constructions used in the proof more natural.

We use the conjugation notation $a^b = b^{-1}ab.$

The notion of first-level sections can be generalized to
any level. If $v\in \tree _m$ is a finite word and $g\in
Aut(\tree _m)$, then there is a word $v'$ of equal length
to $v$ and an automorphism $g'\in\Aut(\tree _m)$ such that
$vw.g=v'(w.g')$, for every word $w$. This $g'$ is called
the {\bf section} of $g$ at $v$. Informally, $g'$ is the
action of $g$ on the subtree above the vertex $v$. The
section of $g$ at $v$ is denoted $g(v)$.

\paragraph{Activity growth of automaton groups.}

For any state $a\in A$, the number of length-$n$ words $v$
such that the section $a(v)$ is not the identity satisfies
a linear recursion. Thus this number grows either
polynomially with some degree $d$ or exponentially. We
define the {\bf degree of activity growth} (in short,
degree) of $a$ to be $d$ or $\infty$, respectively. The
{\bf degree} of an automaton group is the maximal degree of
any of its generators. Automaton groups are said to have {\bf
bounded, linear, polynomial} or {\bf exponential} activity
growth when their degree is $0$, $1$, finite or infinite,
respectively.

\paragraph{Degree and cycle structure.}

An automaton gives rise to a directed graph where there is
a directed edge from a state $a$ to each of its first-level
sections $a(i)$. If the same state appears more than once
as an $a_i$ then there are parallel edges. If $a$ appears
as a first-level section of itself (i.e., $a(i)=a$) then
there is a loop at $a$. Thus the only edges leaving the
identity (if it is a state of $A$) are loops back to the
identity. The loops at the identity are called {\bf trivial
cycles}. To avoid degeneracies in the graph, we will assume
from now on that all non-identity states in the automaton
act non-trivially. This assumption, which does not limit
the generality of the automaton groups considered, allows
some connections between the structure of the said graph
and properties of the automaton.

The number of active vertices of $a$ at level $n$ is just
the number of directed paths of length $n$ starting at $a$
and ending anywhere but the identity. It follows easily
that an automaton is exponential if and only if there are
two nontrivial cycles so that each is reachable from the
other via a directed path.

The nontrivial directed cycles in the directed graph of a
polynomial automaton have a partial order: $c_1 < c_2$ if
$c_1 \neq c_2$ and there is a directed path from a state of
$c_2$ to a state of $c_1$. Define the degree of a cycle $c$
as the maximal $n$ so that there is an increasing chain
$c_0 < c_1 < \ldots < c_n = c$. It is straightforward to
see that the degree of a polynomial automaton is the
maximal degree of any cycle in its directed graph. The
degree of a state $g$ is the maximal degree of a cycle
reachable via a directed path from $g$. The identity $\id$
is always considered to have degree -1. An automaton generates a finite
group if and only if it contains no non-trivial cycle.

\paragraph{Hierarchy levels.} For polynomial automaton,
the hierarchy level of a state is a refinement of its
degree, taking values in the sequence
\[
-\tilde 1, -1,\tilde 0, 0,\tilde 1, 1 ,\tilde 2, 2 , \ldots
\]
  States in cycles of degree $d$ have {\bf hierarchy level
$\tilde d$}. States of degree $d$ that are not in a cycle
have hierarchy level $d$. A
 state can point to states either in its own hierarchy level or lower
levels. The only state with hierarchy level $-\tilde 1$ is the identity.

\paragraph{Collapsing levels.}

By dividing words in $m$ symbols into blocks of length $k$,
we can view them as words in $m^k$ symbols. Similarly,
given an automaton $A$ acting on $m$ symbols, it naturally
gives rise to an automaton with the same states acting on
words in $m^k$ symbols; we call this new automaton the {\bf
$k$-collapse} of $A$, and it acts on the tree $\tree
_{m^k}$.

\begin{claim}
  The automaton groups of $A$ and of its $k$-collapse are isomorphic.
\end{claim}

\begin{proof}
  The key is that the the vertices of $\tree _{m^k}$ are naturally
  associated with every $k$'th level of $\tree _m$, and that $\Aut(\tree
  _m) \subset \Aut(\tree _{m^k})$. Restricting this embedding to
  $G=\langle{A}\rangle$ gives an isomorphism from the automaton group of
  $A$ to the group of the $k$-collapse of $A$.
\end{proof}

\section{Entropy and asymptotic entropy}\label{s:entropy}

The purpose of this section is to review the notion of entropy
and to prove Proposition~\ref{p:entropy_usc} below, which
gives a condition for upper semi-continuity of the
asymptotic entropy of a random walk on a group.

Through the rest of the section we assume $\{\mu_i\},\mu,\nu$ to be
non-negative measures supported on a countable set. Recall that the {\bf
  entropy} of a finite non-negative measure $\mu$ supported on a countable
set $G$ is defined by
\[
\Ent(\mu) = \sum_{x\in G} -\mu(x) \log\mu(x),
\]
where by convention $0\log 0=0$.
The entropy $H(X)$ of a discrete random variable $X$ is
given by the entropy of its distribution; the entropy
$H(X_1,\ldots X_n)$ of more random variables is given by
the entropy of the joint distribution of the $X_i$. In
order to define conditional entropy for two random
variables $X,Y$, let $f(y)$ denote entropy of
the conditional distribution of $X$ given $Y=y$. Then the
{\bf conditional entropy of $X$ given $Y$} is defined as
$H(X|Y) := \E f(Y)$.

The conditional entropy satisfies
\[
H(X,Y) = H(X|Y) + H(Y).
\]
A useful and easy fact is that among measures supported on a given
finite set, the one having maximal entropy is the uniform measure
on that set. Another well known fact, which is relevant to our
cause is that among all measures supported on the nonnegative
integers with given expectation $M$, the entropy is maximized by
the geometric distribution with
$P(\mu=k)=\frac{1}{1+M}\left(\frac{M}{M+1}\right)^k\,$
($k=0,1,2,\ldots$). In particular, this gives the fact, which be
of use later on:

\begin{lemma}\label{l:entropy and expectation}
  For any random variable
  $\tau$ supported on $\N$ we have
  \[
  H(\tau) \leq 2\log (\E \tau + 2).
  \]
\end{lemma}
\vskip0.15in

\noindent Define the {\bf asymptotic entropy} of a sequence of random
variables $X_n$ as
$$
\limsup_{n\to \infty} \frac1n H(X_n).
$$
If a random walk on a group $G$ has i.i.d.\ steps with distribution given
by $\mu$, its asymptotic entropy is given by
\[
\Entlim(\mu) = \lim_{n\to\infty}\frac1n \Ent(\mu^{*n}),
\]
where $\mu^{*n}$ is the $n$-fold convolution of $\mu$, or,
equivalently, the distribution of the $n$'th step of the
random walk. The limit exists by sub-additivity.

Recall that a set of probability measures $\{\mu_i\}$ is {\bf tight} if for
every $\eps$ there exists a finite set $K$ so that for all $i$ $ \
\sum_{x\notin K} \mu_i(x)< \eps.$

We say that a set of probability measures $\{\mu_i\}$ is {\bf
  entropy-tight} if for every $\eps$ there exists a finite set $K$ so that
for all $i$
\begin{equation}\label{e:entropy-tight}
  \sum_{x\notin K} -\mu_i(x)\log \mu_i(x) < \eps.
\end{equation}
In other words, entropy-tightness means the uniform
integrability of the function $\mu_i(x)\log \mu_i(x)$ with
respect to counting measure. The importance of
entropy-tightness comes from the following direct
application of Vitali's convergence lemma.

\begin{lemma}\label{L:vitali}
  Assume that $\mu_n \to \mu$ weakly. Then $\Ent(\mu_n) \to \Ent(\mu)$ if
  and only if the sequence $\mu_n$ is entropy-tight.
\end{lemma}

The aim of this section is the following result:

\begin{proposition}\label{p:entropy_usc}
  If $\mu_i \to \mu$ and $\{\mu_i\}$ is entropy-tight, then
  $\limsup \Entlim(\mu_i) \leq \Entlim(\mu)$.
\end{proposition}

We begin with two lemmas.

\begin{lemma}\label{L:Hconvolution}
  For any two measures $\mu$ and $\nu$ we have $\Ent(\mu*\nu) \leq
  |\nu|\Ent(\mu)+|\mu|\Ent(\nu)$.
\end{lemma}

\begin{proof}
  For the product measure $\mu\times \nu$ we have
  \begin{align*}
    \Ent(\mu \times \nu)
    &= \sum_{x,y} \mu(x)\nu(y) \log(\mu(x)\nu(y)) \\
    &= \sum_{x,y} \mu(x)\nu(y) \log\mu(x)
     + \sum_{x,y} \mu(x)\nu(y) \log\nu(y) \\
    &= |\mu| \Ent(\nu) + |\nu| \Ent(\mu).
  \end{align*}
  Now by sub-additivity of the function $-x\log x$ we have $\Ent(\mu*\nu)
  \leq \Ent(\mu\times \nu)$. The lemma follows.
\end{proof}

\begin{lemma}\label{L:convtight}
  If a family of probability measures $\{\mu_i\}$ is both tight and
  entropy-tight then so is the set of all their $n$-fold convolutions
  $\{\mu_{i_1} * \cdots * \mu_{i_n}\}$.
\end{lemma}

\begin{proof}
  It suffices to prove this for $n=2$, with larger $n$ following by
  induction. For a measure $\mu$ we denote the measure
  restricted to $A$, without normalization, by $\mu^A$.

  Fix some $\eps>0$. The conditions imply that there is a finite set $K$
  such that for all $i$,
  \begin{align*}
    |\mu_i^{K^c}| &= \mu_i(K^c) < \eps,  &  \Ent(\mu_i^{K^c}) &< \eps.
  \end{align*}
  For  any  two measures $\mu,\nu$ from the set we have
  \[
  \mu*\nu = \mu^K * \nu^K + \mu^{K^c} * \nu^K + \mu^K * \nu^{K^c} +
  \mu^{K^c} * \nu^{K^c}.
  \]
  Consider the finite set $B=K*K$, and note that the support of
  $\mu^K*\nu^K$ is contained in $B$. It follows that
  \[
  \left|(\mu*\nu)^{B^c}\right| \leq \left|\mu^{K^c} * \nu^K + \mu^K *
    \nu^{K^c} + \mu^{K^c} * \nu^{K^c} \right|
  < 2\eps + \eps^2,
  \]
  which can be made arbitrarily small, hence the convolutions form a tight
  family.

  Note that for $a,b\geq 0$ we have $-a\log a - b\log b > -(a+b)\log(a+b)$,
  so that entropy of measures is sub-additive. Note also that
  entropy-tightness implies that for some $M$, for every $i$ we have
  $\Ent(\mu_i)\leq M$.

  It follows using \lemref{Hconvolution} that
  \begin{align*}
    \Ent\left((\mu*\nu)^{B^c}\right)
    &\leq \Ent\left( \mu^{K^c} * \nu^K + \mu^K * \nu^{K^c} + \mu^{K^c} *
      \nu^{K^c} \right) \\
    &< 2M\eps + 2\eps + 2\eps^2.
  \end{align*}
  Since this too can be made arbitrarily small, the convolutions are also
  entropy-tight.
\end{proof}

\begin{proof}[Proof of Proposition \ref{p:entropy_usc}]
  Fix some $n$. By \lemref{convtight} we have that $\{\mu_i^{*n}\}_i$ are
  entropy-tight, and so by \lemref{vitali} we have
  \[
  \Ent(\mu^{*n}) = \lim_{i\to\infty}\Ent(\mu_i^{*n}).
  \]

  The sequence $\Ent(\mu_i^{*n})$ is sub-additive, and therefore
  \[
  \Entlim(\mu_i) \leq \frac{1}{n}\Ent(\mu_i^{*n}).
  \]
  Taking the $\limsup$ of both sides we get
  \[
  \limsup_{i\to\infty} \Entlim(\mu_i)
  \leq \frac{1}{n}\limsup_{i\to\infty} \Ent(\mu_i^{*n})
  = \frac{1}{n} \Ent(\mu^{*n}).
  \]
  Since $n$ is arbitrary, we can take a limit as $n\to\infty$ to conclude
  the proof.
\end{proof}

\section{Random walks and the ascension operator}
\label{s:ascension}

In our proof we use the method, introduced by \citet{BV05},
of studying a random walk on an automaton group by looking
at its induced action on the subtree above a vertex $v$,
specifically at times $n$ at which $X_n$ fixes $v$.

\paragraph{Random walks on quotients and subgroups.}

Let $\{X_n\}_{n\geq0}$ be a random walk on a countable
group $G$ started at $X_0=\id$.

If $N$ is a normal subgroup of $G$, and $K=G/N$, then $X$
has a canonical projection to $K$, namely the cosets $Y_n = NX_n$. We
call this the {\bf quotient random walk} on $G/N$.

A set $S\subset G$ is called recurrent (for a random walk
$X$) if $X$ visits $S$ infinitely often with probability
one. For example, every finite index subgroup is recurrent
in $G$. For a recurrent subgroup $S$ the steps at which the
random walk visits $S$ form a random walk $Y$ on $S$ (i.e.\
if $\tau_n$ is the $n$th visit to $S$, let
$Y_n=X_{\tau_n}$). We call this the {\bf induced random
walk on $S$}.

\paragraph{The ascension operator.}

Given an automaton group $A$ and a vertex $j\in\tree_m$,
consider the stabilizer subgroup $A_j$ of $j$. Consider
also the subgroup $A'_j$ which is the stabilizer of the
entire subtree above vertex $j$. Then $A'_j$ is normal in
$A_j$. For $g\in A_j$ the coset $g A'_j$ consists of all
elements with the same action $g(j)$ on the subtree above
$j$. Since $A$ is an automaton group, $g(j)\in A$, hence
the group $A_j/A'_j$ is canonically isomorphic to
a subgroup of $A$.

We now specialize to vertex $j=0$ in the first level of
$\tree_m$. Given a random walk on $A$ with step
distribution $\mu$, we can consider the induced walk on
$A_0$, and then its quotient walk on $A_0/A'_0$. By the
above, this again can be viewed as a random walk on $A$
with step distribution $\mu'$. The {\bf ascension operator}
$T$ is the operator that maps each probability measure
$\mu$ on $A$, to the measure $\mu'$ above. If $\tau_n$ are
the times at which $X_{\tau_n}$ fixes vertex 0, then the
actions of $X_{\tau_n}$ on the subtree above $0$ are a
random walk with step distribution $T \mu$.

We say that a measure $\mu$ is transitive on level $k$ if
the group generated by its support acts transitively on
that level of $\tree_m$. We will use the following entropy
inequality.

\begin{theorem}[\citet{Kaimanovich05}, Theorem~3.1]\label{T:AE down}
  Assume that a probability measure $\mu$ is transitive on the first level.
  Then the asymptotic entropies satisfy $\Entlim(\mu) \le \Entlim(T\mu)$.
\end{theorem}

The upper semi-continuity of asymptotic entropy,
Proposition \ref{usc H} yields the following.

\begin{theorem}[Asymptotic entropy of automaton groups]\label{maintool}
  If the group generated by the support of $\mu$ is transitive on all
  levels and the sequence $\{T^k \mu\}$ is entropy-tight, then for any
  subsequential limit point $\nu$ we have
  \[
  \Entlim(\mu) \le \Entlim(\nu).
  \]
\end{theorem}

\begin{proof}
  The transitivity of $\mu$ implies that $T^k\mu$ is transitive on the
  first level for all $k$. Repeated application of Theorem~\ref{T:AE
    down} shows that for each $k$
  \[
  \Entlim(\mu) \le \Entlim(T^k\mu).
  \]
  Taking limsup along the subsequence converging to $\nu$ and using
  Proposition~\ref{p:entropy_usc} gives
  \[
  \Entlim(\mu) \leq \limsup \Entlim(T^k \mu) \leq \Entlim(\nu).
  \qedhere
  \]
\end{proof}

\section{Mother groups}\label{s:mother}

The {\bf $(m,r)$-mother group}, denoted $\M_{m,r}$, is defined as
the group generated by the automaton with the following states
\begin{align}
  \alpha_{k,\sigma} &= \llangle  \alpha_{k,\sigma},
  \alpha_{k-1,\sigma},1,\ldots,1\rrangle , \qquad 0\le k \le r,  \nonumber\\
  \alpha_{-1,\sigma} &= \sigma   \label{eq:mother_generators}\\   
  \beta_{k,\rho} &= \llangle  \beta_{k,\rho},
  \beta_{k-1,\rho},1,\ldots,1\rrangle ,\qquad 1\le k \le r,\nonumber\\
  \beta_{0,\rho} &= \llangle
  \beta_{0,\rho},1,\ldots,1\rrangle \rho.\nonumber
\end{align}
where $\sigma,\rho \in \Sym(m)$ are arbitrary, subject to $0.\rho
= 0$. number of states in $\M_{m,r}$ as defined here is
$m!(r+2)+(m-1)!(r+1)$. The same group can be generated by a
smaller set of elements by taking $\sigma,\rho$ only in a minimal,
$2$-element set of generators of $\Sym(m)$ and
$\stab(0)\subset\Sym(m)$ respectively. This would give a
generating set of size $4r+6$. However, our original choice will
prove more suited to our purpose.

The actions of $\alpha_{k,\sigma}$ and $\beta_{k,\rho}$ on a word have
simple descriptions. Both read the word and make no changes up to the
$k+1$'th non-zero letter.
\begin{itemize}
\item If the first $k+1$ nonzero letters in a word are all 1, then
  $\alpha_{k,\sigma}$ permutes the next letter by $\sigma$. Otherwise it
  does nothing.
\item If the first $k$ nonzero letters in a word are all 1, then
  $\beta_{k,\rho}$ permutes the next nonzero letter by $\rho$. Otherwise it
  does nothing.
\end{itemize}
Thus both affect only the $k+1$'st non-zero letter and the
letter immediately following it.

Note that $\alpha_{k,\sigma}$ and $\beta_{k,\rho}$ both have self-loops and
are of degree $k$, so they have hierarchy level $\tilde k$.

\begin{theorem}[Mother groups contain all]\label{T:universal_mother}
  Every degree-$r$ automaton group is isomorphic to a subgroup of
  $\M_{m,r}$ for some $m$.
\end{theorem}

Note that $m$ is generally not the same as the degree of the tree on which
the automaton acts.

\begin{proof}
  We prove that there exist $m',m''$ so that the each of the following
  groups can be isomorphically embedded in the next: $A \subset G_{m',r}
  \subset \M^*_{m',r} \subset \M_{m'',r}$. The intermediate groups are
  defined below. The three containments are proved in Lemmas~\ref{l:A
    visits grandmother}, \ref{l:grandmother visits mother*}, \ref{l:mother*
    visits mother}.
\end{proof}

\subsection*{The first reduction}

For each hierarchy level $h\in\{-\tilde 1,-1,\tilde 0,
0,\tilde 1, 1 ,\tilde 2, 2 , \dots\}$ we define an
automaton group as follows. $G_{m,-\tilde 1} = \{\id\}$,
and $G_{m,-1}= \Sym(m)$. For any $r\geq0$ the group $G_{m,
\tilde r}$ is obtained by adding to $G_{m,r-1}$ all
possible elements of the form:
\[
g = \llangle  g_0,\ldots, g_{m-1}\rrangle  \sigma
\]
(with $\sigma \in \Sym(m)$) satisfying the following. There
is a unique $i$ such that $g_i=g$, and for $j\neq i$ we
have that $g_j$ is an element of $G_{m, r-1}$.

Similarly, for $r\geq0$ we define $G_{m, r}$ by adding to $G_{m,\tilde r}$
all possible elements of the form:
\[
g = \llangle  g_0,\ldots, g_{m-1}\rrangle  \sigma
\]
where $\sigma \in \Sym(m)$, and $g_j$ is an element of $G_{m,
\tilde r}$ for all $j$.

Since they are larger than the mother groups, and are also
predecessors of the mother groups, in this section we will refer
to the groups $G_{m,h}$ as {\bf grandmother groups}.

\begin{lemma}\label{l:A visits grandmother}
  Let $A$ be an automaton with all states having hierarchy level at most
  $h$. Then for some $k$, the automaton group corresponding to the $k$-collapse of $A$ is isomorphic to a
  subgroup of $G_{m^k,h}$
\end{lemma}

\begin{proof}
  Let $k\in\Z$ be a multiple of all cycle lengths of $A$ and larger than
  the length of any {\em simple} directed path in $A$.
  The $k$-collapsed version $A'$ of $A$ has the property that all of its
  cycles are loops. States with loops have a single loop only, since
  otherwise $A$ would have been exponential.

  A state of a given hierarchy level can only point to a state from a lower
  level, except for states with loops that also point to themselves. By
  construction, the grandmother group $G_{m',r}$ contains all possible
  elements of this form.
\end{proof}

\subsection*{Extended mother groups}

Let $W_k$ denote the finite subgroup of $\Aut(\tree_m)$ for
which all sections at level $k$ are the identity.

\begin{definition}
  The {\bf extended mother groups} $\M_{m,\widetilde r}^*, \M_{m,r}^{*}$
  are defined the same way as the ordinary mother groups $\M_{m,r}$, except
  that in the definition of the $\alpha_{\sigma,\ell}$ states
  \eqref{eq:mother_generators} ,$\sigma$ ranges over all elements of
  $W_{r+2}$ (for $\M_{m,\widetilde r}^*$)  and $W_{r+3}$ (for
  $\M_{m,r}^{*}$). (The definition of the $\beta_{\rho,\ell}$ states
  remains unchanged).
\end{definition}

The extended mother groups are nested:
\[
\M^*_{m,k-1} \subset \M^*_{m,\widetilde k} \subset
\M^*_{m,k}.
\]
Moreover, we have the following lemma.

\begin{lemma}\label{l:r to r tilde}
  If $g_0,\dots,g_{m-1}\in \M^*_{m,\tilde r}$, then
  \[
  g=\llangle  g_0,\ldots , g_{m-1}\rrangle  \in \M^*_{m,r}.
  \]
\end{lemma}

\begin{proof}
  By taking products, it suffices to prove this for the case when all $g_i$
  are the identity except for one, which is a state of $\M^*_{m,\tilde
    r}$. Moreover, since $\Sym(m)\subset\M^*_{m,r}$, by conjugating by a
  transposition $\tau \in \Sym(m)$, we may assume that the non-identity
  entry is $g_0$.

  We now prove the claim by induction on the degree of $g_0$. Consider
  first elements of type $\alpha$. For $g_0=\alpha_{-1,\sigma}=\sigma$ this
  holds by the definition of $\M^*_{m,r}$ (this is the reason for the
  choice of $W_{r+3}$ in the definition). For higher degree states, by
  definition
  \[
  \alpha_{k,\sigma} = \llangle \alpha_{k,\sigma}, \alpha_{k-1,\sigma}, \id,
  \dots, \id \rrangle,
  \]
  so
  \[
  \llangle \alpha_{k,\sigma}, \id, \dots, \id \rrangle = \alpha_{k,\sigma}
  \llangle \id, \alpha_{k-1,\sigma}, \id, \dots,\id \rrangle ^{-1}.
  \]
  Now, $\alpha_{k,\sigma}\in \M^*_{m,\tilde r}\subset \M^*_{m,r}$, and the
  induction hypothesis implies $\llangle \id, \alpha_{k-1,\sigma}, \id,
  \dots, \id \rrangle \in \M^*_{m,r}$ as well, hence so is their product.

  The proof for type $\beta$ states is trivial since they their definition remained the same as in the original mother groups.
\end{proof}

\begin{lemma}\label{l:r-1 to rtilde}
  If $g_1,\dots, g_m \in \M^*_{m,r-1}$, and $\rho\in\Sym(m)$ has
  $0.\rho=0$ then
  \[
  a= \llangle a, g_1,\ldots , g_m\rrangle \rho \in \M^*_{m,\widetilde{r}}.
  \]
\end{lemma}

\begin{proof}
  If $g$ is a generator of $\M^*_{m,r-1}$ then by definition
  $a=\llangle a,g,1,\dots,1\rrangle \in \M^*_{m,\widetilde{r}}$. Multiplying
  such generators shows that this is true for any $g\in\M^*_{m,r-1}$.
  Define $a_i = \llangle a_i, g_i,1,\dots\rrangle$. We have
  \[
  a=\left( \prod_i a_i^{(1i)} \right) \rho
  \]
  where $(1j)$ is the transposition. Since $\beta_{0,(1j)}$ and
  $\beta_{0,\rho}$ are in $\M^*_{m,\widetilde{r}}$, it follows that $a\in
  \M^*_{m,\widetilde{r}}$ as claimed.
\end{proof}

The key step in the following construction is a conjugation due to
\citet{BKN}, where it was used for bounded automaton
groups. Consider the automorphism
\[
\delta = \ang{\delta ,\delta\gamma^{-1},\delta
\gamma^{-2},\ldots }\sigma
\]
where $\gamma$ is the cyclic shift $(0 1 2 \ldots m-1)$.

\begin{lemma}\label{l:grandmother visits mother*}
  For every hierarchy level $h$ the conjugated grandmother group
  $G_{m,h}^\delta$ is a subgroup of the mother group $\M^*_{m,h}.$
\end{lemma}

\begin{proof}
  The proof proceeds by induction on the $h$. On each level, it suffices to
  show this for the group generators. For hierarchy level $-1$ we have
  \[
  \sigma^{\delta} = \ang{\gamma^{0-0.\sigma}, \ldots
    ,\gamma^{m-1-(m-1).\sigma}} \sigma
  \]
  is an element of $W_2 = \M^*_{m,-1}$.

  Assume the lemma holds for hierarchy level $r-1$. Let $g=\ang{g_0,\ldots
    g_{m-1}}\sigma$ be a generator of $G_{m,\widetilde{r}}$. Assume that
  $g_k=g$ is the loop at element $g$. We need to prove that $g^\delta$ is an
  element of $\M^*_{m,\widetilde r}$, for which it suffices to show that
  this holds for $ h = \gamma^k g^\delta \gamma^{-k.\sigma}. $ We have
  \[
  h = \ang{h_0,\ldots, h_{m-1}}\gamma^k\sigma \gamma^{-k.\sigma},
  \qquad \mbox{where} \qquad
  h_i = \gamma^{i+k} g_{i+k}^\delta \gamma^{-(i+k).\sigma}.
  \]
  Now $h_0=h$, the $h_1,h_2,\ldots $ are in $\M^*_{m,r-1}$, and
  $\gamma^k\sigma \gamma^{-k.\sigma}$ fixes $0$. Thus by Lemma~\ref{l:r-1
    to rtilde} we have that $h\in \M^*_{m,\widetilde{r}}$.

  To get from hierarchy level $\widetilde r$ to hierarchy level $r$,
  suppose that $g=\ang{g_0,\ldots g_{m-1}}\sigma$ where $g_i$ are
  generators of $\M^*_{m,\widetilde{r}}$. Then
  \[
  g^\delta = \ang{g_0^{\delta \gamma^0}, \dots, g_{m-1}^{\delta
      \gamma^{m-1}}} \sigma^\delta \in \M^*_{m,r}
  \]
  by the inductive hypothesis and Lemma \ref{l:r to r tilde}.
\end{proof}

\subsection*{Embedding in $\M$}

The extended mother groups contain the ordinary mother groups, but there
are also embeddings in the other direction, as the following lemma shows.

\begin{lemma}\label{l:mother* visits mother}
  The $r+3$-collapsed version of $\M^*_{m,r}$ is a subgroup of
  $\M_{m^{r+3},r}$.
\end{lemma}

\begin{proof}
  It suffices to show inclusion of the $r+3$-collapsed generators. Let
  $m'=m^{r+3}$.
  Clearly, the $r+3$-collapse of
  any element in $W_{r+3}$ is just an element in $\Sym(m')$. The
  $r+3$-collapsed version of any other state of type $\alpha$ is
  of the form
  \[
  a=\llangle  a,g_1,\ldots, g_{m'-1}\rrangle  \sigma
  \]
  where $\sigma\in \Sym(m')$ and the $g_i$'s are states
  of lower degree than $a$. Just as in Lemma~\ref{l:r-1 to rtilde}, we
  define $a_i$ from $g_i$ by
  \[
  a_i = \llangle  a_i,g_i,\id,\ldots,\id \rrangle
  \]
  and we note that $a_i$ is a state of $\M_{m',r}$. $\llangle
  \id,\id,\ldots,\id \rrangle\sigma$ is just $\alpha_{-1,\sigma}$ and so is
  also in $\M_{m',r}$. Thus we can conjugate by the transposition
  $(1i)$ and write
  \[
  a = \left( \prod_{i=1}^{m-1} a_i^{(1i)} \right) \sigma
  \]
  as an element of $\M_{m',r}$.  The proof for states of type $\beta$
  follows similar arguments.
\end{proof}

\subsection*{Level subgroups}

Observe that the group of automorphisms of the first two levels of $\tree
_m$ fixing $0$ and its children is isomorphic to $\Sym(m) \wr \Sym(m-1)$.
(We will interpret elements in $\Sym(m-1)$ as acting on
$\{1,\ldots,m-1\}$.)

For each $\sigma\in \Sym(m) \wr \Sym(m-1)$ and each word $w$ in the symbols
$\{1,\dots,m-1\}$ let $\lambda_{w,\sigma}$ denote the element of $\Aut(\tree _m)$
acting as follows: If the first $|w|$ nonzero letters agree with $w$, then
$\lambda_{w,\sigma}$ permutes the $|w|+1$st nonzero letter and the
following letter by $\sigma$. Otherwise $\lambda_{w,\sigma}$ does nothing. ( e.g. $\lambda_{21,(01)\wr (12)}(\cdots 001020010) = \cdots 012020010$ and $\lambda_{21,(01)\wr (12)}(\cdots 002010010) = \cdots 002010010$)

For a word $w$ of length $k$, define the group $\L^w_{m,k}$
generated by $\lambda_{w,\sigma}$ as $\sigma$ ranges over
$\Sym(m)\wr\Sym(m-1)$. Define the group $\L_{m,k}$ to be the group
generated by the $\L^w_{m,k}$ for all words $w$ of length
$k$. Define further $\L_{m,-1}=\Sym(m)$.

Later, we will consider random walks on the mother group whose step
distribution is a convex combination of uniform measures on the subgroups
$\L_{m,k}$ for various $k$'s.

\begin{lemma}
  For each $w$, $\L^w_{m,k} \approx \Sym(m)\wr\Sym(m-1)$. The
  group $\L_{m,k}$ is a subgroup of
  $\M_{m,k}$ and is the direct product of $\L^w_{m,k}$ for
  $w\in\{1,\dots,m-1\}^k$. Moreover, the mother group $\M_{m,k}$ is generated by
the subgroups $\{\L_{m,\ell}\}_{\ell\leq k}.$
\end{lemma}

\begin{proof}
  The structure of $\L^w_{m,k}$ follows from $\lambda_{w,\sigma}
  \lambda_{w,\sigma'} = \lambda_{w,\sigma\sigma'}$. Generators
  corresponding to different words $w$ of equal length commute, hence
  $\L_{m,k}$ is the direct product of the $\L^w_{m,k}$'s.

  For $w = 11\dots1$ of length $k$ we have that $\L^{ w}_{m,k}$ is
  generated by the $\beta_{k,\rho}$ and
  $\alpha_{k,\sigma}$'s, hence it is a subgroup of $\M_{m,k}$.

  Now let $w$ be a general word of length $k$, and note that the
  automorphism
  \[
  b=\beta_{k-1,(1w_k)}\cdots \beta_{0,(1w_1)}
  \]
  changes only the first $k$ non-zero letters, and if they are all $1$ it
  changes them to the letters of $w$. Thus $\L^w_{m,k}$ is the conjugate by
  $b$ of $\L^{11\dots1}_{m,k}$, and therefore $\L^w_{m,k} \subset \M_{m,k}$
  for any $w$.

Note that all elements in $\L_{m,\ell}$ have degree at most
$\ell$. Also since for any $\sigma,\rho\in \Sym(m)$ with
$\rho.0=0$ the generators $\alpha_{\ell,\sigma},\beta_{\ell,\rho}$
are in $\L_{m,\ell}$, the mother group $\M_{m,k}$ is generated by
the subgroups $\{\L_{m,l}\}_{\ell\leq k}.$
\end{proof}

From now on, we will use the shorthand notation $\L_k$ and $\M_k$
for $\L_{m,k}$ and $\M_{m,k}$ respectively.

\section{Patterns}
\label{s:patterns}

The random walks on the mother groups that we consider below have
a step distribution that is a mixture of the uniform measures
$\bar{q_i}$ on the finite level subgroups $\L_i$, and
convolutions of these measures. It is convenient to think of the measures
$\bar{q_i}$ as elements of the group algebra $\R \M_k$. Note that
every element of $g\in\L_i$ for $i\geq 0$ is of the form
$g=\llangle g,g_1,\ldots g_m\rrangle$ with $g_j\in \L_{i-1}$.
Choosing an element of $\L_i$ uniformly corresponds to choosing
the $g_j$-s uniformly from $\L_{i-1}$. This property is central to
the control of the ascension of measures that we study.

The elements $\bar q_i\in \R \M_k$ satisfy the relations
$\bar q_i^2=\bar q_i$ (being uniform measures on
subgroups), and possibly some other relations that are less
tractable. Therefore we introduce the more basic semi-group
(with identity $\emptyset$) $\cQ_k$ given by the generators
$q_{-1},\ldots, q_k$, (which we call {\bf pattern letters})
and the relations $q_i^2 = q_i$. Elements of $\cQ_k$ are
called {\bf patterns}. We further define $D=D_k\subset
\cQ_k$ to be the set of patterns that contain the letter
$q_{-1}$.

Since the measures $\bar q_i$ satisfy all relations satisfied by the $q_i$,
the map $q_i\mapsto \bar{q_i}$ extends multiplicatively to a unique
semi-group homomorphism $\cQ_k\to \R \M_k$. The image $\bar p$ of a pattern
$p$ is called the {\bf evaluation} of $p$.

Each pattern is an equivalence class of words in the letters
$\{q_{-1},q_0,\dots,q_k\}$. The equivalence relation is that repetition of
a symbol is equivalent to a single instance. For example, for $a,b,c$ in
the set of letters, we have $abaacaaabba\equiv abacaba$ . Composition is
concatenation. The {\bf length} of a pattern $p$, i.e.\ the length of the
shortest element of its equivalence class is denoted $|p|$. The set $\cQ_0$
is of particular interest to us. It contains patterns in two letters, which
must alternate. Thus $\cQ_0$ has only two patterns of any length. Finally,
for a measure $\mu$ on patterns and $\alpha\ge 0$ the {\bf $\alpha$-moment}
of the length is denoted by
\[
\mu(|p|^\alpha) = \sum_{p\in \cQ_k} |p|^\alpha\mu(p).
\]

The main reason for the definition of patterns is that they behave
nicely with respect to the ascension operator. To make this statement more
precise, consider a probability measure $\mu$ supported on patterns. $\mu$
is a convex combination of measures concentrated on a single pattern, and
is naturally an element of $\R \cQ_k$. Its image $\bar{\mu}$ under the
quotient map is an element of $\R \M_k$, or more precisely, a probability
measure on $\M_k$. The measure $\bar{\mu}$ is will also be called the {\bf
  evaluation} of $\mu$. It is given by the formula
\[
\bar{\mu} = \sum_{p\in\cQ_k} \mu(p) \bar{p}.
\]

Probability measures of the form $\bar{\mu}$ are special. Nevertheless,
this class is preserved by the ascension operator $T$. This suggests that
the ascension operator can be defined on the level of patterns.

\subsection*{The ascension formula}

Let $\{Z_n\}$ be the random walk on $G$, with steps $X_n$,
so that $Z_n=X_1 \cdots X_n$, and write $X_n = \llangle
X_n(0),\dots,X_n(m-1) \rrangle \sigma_n$. Clearly the
trajectory of $0$, is itself a Markov chain: If $J_n = 0.Z_n$
then $J_n = J_{n-1}.\sigma_n$.
The section $Z_n(0)$ is not itself a Markov chain (except
for very restricted step distributions). However, the pair
$(J_n,Z_n(0))$ is a Markov chain, since $Z_n(0) =
Z_{n-1}(0) X_n(J_{n-1})$: to determine the action of $Z_n$
on the subtree above $0$ we need to know the action of
$Z_{n-1}$ on that subtree, and a single section of $X_n$,
with index given by $0.Z_{n-1}$.

A key observation used by \citet{Kaimanovich05} is that if
we consider the process $\{Z_n(0)\}$ conditioned on
$\{J_n\}$ we get a process with independent increments.
Moreover, if $Z_n(0) = Y_1\cdots Y_n$, then $Y_n$ depends
on $\{J_n\}$ only through the pair $J_{n-1},J_n$. Indeed
the law of $Y_n$ is the distribution of $Z_n(J_{n-1})$
conditioned on $J_{n-1}.\sigma_n = J_n$.

Such a process can be described naturally by the $m\times m$ matrix $M$
with entries in the group algebra $\R G$, with entries
\begin{equation}
  \label{eq:M_ij_def}
  M_{ij} = \E X(i) 1_{\{i.X=j\}}.
\end{equation}
Here $X\in\M_k$ is a step of the original random walk, and
therefore the section $X(i)$ is also in the group. To take
its expectation consider $X(i)$ as an element of the group
algebra $\R \M_k$. Note that the entries of this matrix are
interpreted as sub-probability measures. Their total mass
$\|M_{ij}\|_1$ are the transition probabilities for the
Markov chain $\{J_n\}$. A useful consequence of this
definition is that for $\mu=\mu_1*\mu_2$ we have $M = M_1
M_2$.

For an $m\times m$ matrix $M$, consider its block decomposition
\[
  M = \left( \begin{array}{c|c} M_{00} & M_{0*} \\ \hline M_{*0} &
    M_{**} \end{array} \right),
\]
so that $M_{**}$ is an $(m-1)\times(m-1)$ matrix.

\begin{proposition}\label{p:asc_form}
  Let $M$ be as above, then
  \[
  T\mu = M_{00} + M_{0*} (I-M_{**})^{-1} M_{*0},
  \]
  where $(I-A)^{-1}$ means $I+A+A^2+\cdots$.
\end{proposition}

This appears as Theorem~2.3 in \citet{Kaimanovich05}. Note that $M_{**}$ is
an $(m-1)\times(m-1)$ matrix with entries in $\R G$. We include the above
discussion and the brief proof because similar ideas are used next for
measures on patterns.

\begin{proof}
  With the above notation, consider the trajectory of $0$ until the first
  time $\tau$ such that $J_\tau = 0$. The probability of a particular
  trajectory $0=J_0,J_1,\dots,J_\tau=0$ is a product of transition
  probabilities $\|M_{ij}\|_1$ of the $J$ Markov chain. Conditioned on
  these values, the process $Y_n$ is a product of independent samples from
  the probability measures $M_{ij}/\|M_{ij}\|_1$ corresponding to the
  transitions. Paths of length $\tau$ correspond to the term $M_{0*}
  M_{**}^{\tau-2} M_{*0}$, and $\tau=1$ gives the $M_{00}$ term.
\end{proof}

\subsection*{Ascension formula for patterns}

Suppose now that $\mu$ is a probability measure on $\cQ_k$,
and consider the random walk on $\M_{k}$ with step
distribution $\bar{\mu}$. This measure has a great deal of
symmetry. A uniformly chosen term $X \in \L_{-1}$ sends any
$i$ to a uniform vertex on level $1$ of the tree, and has
trivial sections $X(i)$. An element $X\in \L_\ell$ for $\ell>0$
stabilizes the first level of the tree, and its sections
are $X(0)=X$ and $X(i)\in \L_{\ell-1}$, again uniformly.
Similarly, an element $X\in \L_0$ fixes $0$, permutes
$\{1,\ldots, m-1\}$ uniformly, and its sections are
$X(0)=X$ and $X(i)\in \L_{-1}$, again uniformly and
independently.

Thus the matrix $M$ defined above is as follows:
\begin{itemize}
\item For $\bar{q_\ell}$, $\ell>0$ we get a diagonal matrix, with terms
  $M_{00} = \bar{q_\ell}$ and $M_{ii} = \bar{q_{\ell-1}}$.
\item For $\bar{q_{-1}}$ it is the constant matrix $M_{ij} = 1/m$ (entries
  are the measure of mass $1/m$ on the identity element of the group).
\item For $\bar{q_0}$ we have $M_{00}=\bar{q_{-1}}$, and $M_{0i}=M_{i0}=0$
  for $i\neq0$. The remaining minor has constant entries: $M_{ij} =
  \frac{1}{m-1}\bar{q_{-1}}$ for $i,j>0$.
\end{itemize}
Note that in all cases, all entries $M_{ij}$ are in $\R\bar{\cQ_k}$.

Next, consider a measure $\bar{p}$ for a pattern $p=q_{\alpha_1}\cdots
q_{\alpha_\ell}$. Since $\bar{p}$ is a convolution of $\bar{q_{\alpha_i}}$,
the resulting $M$ is a product of the corresponding $M$'s, and also has
entries in $\R \cQ_k$. Finally, from \eqref{eq:M_ij_def} we see that $M$ is
linear in the step distribution. Since $\bar{\mu}$ is a convex combination
of $\bar{p}$, we see that $\bar{\mu}$ also gives $M$ with entries in $\R
\cQ_k$.

Further more, note that for all $\bar{q_\ell}$ we have that the following
sets of entries are all constant:
\begin{itemize}\addtolength{\itemsep}{-0.5\baselineskip}
\item the first row, except $M_{00}$;
\item the first column, except $M_{00}$;
\item the main diagonal, except $M_{00}$;
\item $M_{i,j}$ for $i\neq j$ and $i,j>0$.
\end{itemize}
Note that the algebra of $m\times m$ matrices satisfying the
bulleted conditions preserves the two-dimensional space of vectors
of the form $(x, y, y, ..., y)$. Let $R$ be the operator $M$
acting on this space written in the basis  $\{(1, 0, 0, ..., 0),
(0, 1, 1, ..., 1)\}$. More explicitly,
\begin{align*}
  R_{00} &= M_{00},  &
  R_{01} &= \sum_{j>0} M_{0j},  \\
  R_{10} &= M_{10},  &
  R_{11} &= \sum_{j>0} M_{1j}.
\end{align*}
where $M$ is defined by \eqref{eq:M_ij_def} for a sample of
$\bar{\mu}$. (Taking any $i>0$ in place of $1$ will not make a
difference.) Matrices with the above  properties form an algebra,
and therefore the same identities hold for any $\bar{\mu}$.

Since $M\mapsto R$ is a homomorphism of matrix algebras,
$\mu\mapsto R$ is a homomorphism from $\R \cQ_k$ to the matrix
algebra of $2\times2$ matrices over $\R\M_k$: if $\mu=\mu_1*\mu_2$
then $R=R_1 R_2$.

We now introduce the operators $\cR_{ij}$, so that for any
measure $\mu$ on $\cQ_k$ we have $R_{ij} \bar{\mu} =
\bar{\cR_{ij} \mu}$. Since $\mu\mapsto R$ is a
homomorphism, it suffices to define $\cR$ for patterns of a
single letter, and extend it by multiplicativity and
linearity to all of $\R\cQ_k$. It is clear that the
following definition satisfies this desire. For $k\geq 0$,
define
\[
\cR_{ij} q_k = \begin{cases}
  q_k & i=j=0, \\ q_{k-1} & i=j=1, \\ 0 & i\neq j.
\end{cases}
\]
Note that $k=0$ no longer needs special care (an advantage
of using $R$ over $M$). For $k=-1$ define
\[
\cR_{ij} q_{-1} = \begin{cases} \frac1m & j=0, \\ \frac{m-1}m & j=1.
\end{cases}
\]
It follows that the total mass of $\cR_{ij}\nu$ for any
measure on patterns is given by
  \begin{equation}\label{totalmass}
  {\cR}_{ij} \nu (\cQ_k) = 1_{i=j}(1-\nu(D)) +
  \nu(D)\frac{(m-1)^j}m,
  \end{equation}
where $D$ is the set of patterns containing the letter
$q_{-1}$.  Let $\cR\mu$ to be the matrix with entries
$\cR_{ij} \mu$. As noted above, for a pattern
$p=q_{\alpha_1}\cdots q_{\alpha_\ell}$ we have, using
matrix multiplication, $\cR p = \prod \cR q_{\alpha_i}$.
For measures $\mu$ we define $\cR\mu = \sum_p \mu(p) \cR
p$.

A key consequence of this construction is that a version of
Proposition~\ref{p:asc_form} holds with $\cR$. Define the {\bf pattern
  ascension} operator by
\begin{equation}
  \label{eq:pT_def}
  \pT \mu = \cR_{00} \mu + (\cR_{01} \mu) \big(1 - \cR_{11} \mu
  \big)^{-1} (\cR_{10} \mu),
\end{equation}
where $(1-a)^{-1} = 1 + a + a^2 + \cdots$.

\begin{proposition}\label{p:pat_asc}
  For a probability measure $\mu\in\R\cQ_k$ we have $T \bar{\mu} = \bar{\pT
    \mu}$.
\end{proposition}

\begin{proof}
  Let $M$ be the matrix corresponding to $\bar{\mu}$. By
  Proposition~\ref{p:asc_form},
  \[
  T \bar{\mu} = M_{00} + M_{0*} \big(I - M_{**}\big)^{-1} M_{*0}.
  \]
  However, since $M$ is in the aforementioned matrix algebra, it is
  straightforward to verify that $R_{01} R_{11}^k R_{10} = M_{0*}M_{**}^k
  M_{*0}$, and hence
  \[
  T \bar{\mu} = R_{00} + R_{01} \big(I - R_{11}\big)^{-1} R_{10}.
  \]
  Since $R\bar{\mu} = \bar{\cR\mu}$, this implies the proposition.
\end{proof}

We now inspect the structure of $\cR$ in more detail. Since $\cR$ is
multiplicative, for some patterns its entries are delta measures:

\begin{lemma}\label{L:cR_structure}
  For a pattern $p=q_{\alpha_1}\cdots q_{\alpha_\ell}$ that does not
  contain the letter $q_{-1}$, we have
  \[
  \cR_{ij} p = \begin{cases} p & i=j=0, \\
    q_{\alpha_1-1}\cdots q_{\alpha_\ell-1} & i=j=1, \\ 0 & i\neq j.
  \end{cases}
  \]

  For a general pattern with decomposition $p=p_0 q_{-1} p_1 \cdots q_{-1}
  p_\ell$ into patterns $p_i\in D^c$ we have
  \[
  \cR_{ij} p = (\cR_{ii} p_0) \prod_{i=1}^{\ell-1}
  \left( \frac1m \cR_{00} p_i + \frac{m-1}{m} \cR_{11} p_i \right)
  (\cR_{jj} p_\ell).
  \]
\end{lemma}

\begin{proof}
  Since $\cR$ is multiplicative, the first part holds by definition.
  The second part follows from the first and from $\cR_{ij} q_{-1} =
  (m-1)^j/m$.
\end{proof}

Define the {\bf averaged legacy} operator
\[
\cA = \sum_{i,j\in \{0,1\}} \frac{(m-1)^i}{m} \cR_{ij},
\]
so that $\cA\mu = \begin{pmatrix} \frac1m & \frac{m-1}{m} \end{pmatrix}
(\cR\mu) \oneone$. Note that when $\mu$ is a probability measure, so are
$\sum_j \cR_{ij}\mu$, and $\cA\mu$.

\begin{lemma}
  Let $p = p_0\, q_{-1} \,p_1 \,q_{-1} \,\dots\, q_{-1} \,p_k$ where
  $p_i\in D^c$ be the decomposition of $p$ into subpatterns not containing
  $q_{-1}$ (possibly $p_0$ or $p_k$ are $\emptyset$). Then
  \begin{equation} \label{averaged legacy is a product}
    \cA p = (\cA p_0) \cdots (\cA p_\ell).
  \end{equation}
\end{lemma}

\begin{proof}
  This follows from the fact that $\cR$ is multiplicative. Writing
  $
  \cR q_{-1} = \oneone \begin{pmatrix} \frac1m & \frac{m-1}{m} \end{pmatrix}
  $
  we have
  \begin{align*}
    \prod_{i=0}^\ell \cA p_i
    &=  \begin{pmatrix} \frac1m & \frac{m-1}{m} \end{pmatrix} (\cR p_0)
    \oneone \begin{pmatrix} \frac1m & \frac{m-1}{m} \end{pmatrix} (\cR p_1)
    \oneone \cdots \begin{pmatrix} \frac1m & \frac{m-1}{m} \end{pmatrix}
    (\cR p_\ell) \oneone \\
    &= \begin{pmatrix} \frac1m & \frac{m-1}{m} \end{pmatrix} (\cR p_0)
    (\cR q_{-1}) (\cR p_1) \cdots (\cR q_{-1}) (\cR p_\ell) \oneone \\
    &= \cA p.  \qedhere
  \end{align*}
\end{proof}

\subsection*{Algorithmic description of $\pT$}

We now present an algorithm that uses samples of $\mu$ to get a sample from
$\pT \mu$. This algorithm is an interpretation of
Proposition~\ref{p:pat_asc}, though it is also possible to use it to define
$\pT$, and derive Proposition~\ref{p:pat_asc} from it analogously to
Proposition~\ref{p:asc_form}.

To sample $\pT \mu$ we consider the Markov chain $\{\Eps_n,Q_n\}_{n>0}$
with state space $\{0,1\}\times \cQ_k$ and transition probabilities
\begin{equation} \label{eq:R_transition}
  \P(\Eps_n=j, Q_n=p |\Eps_{n-1}=i, Q_{n-1}) = (\cR_{ij} \mu)(p).
\end{equation}

\begin{lemma}
  If $\Eps_0=0$ and $\tau>0$ is minimal such that $\Eps_\tau=0$, then
  $Q_1\cdots Q_\tau$ has the law of $\pT\mu$.
\end{lemma}

\begin{proof}
  This is almost identical to Proposition~\ref{p:asc_form}. The term
  $\cR_{00}\mu$ corresponds to $\tau=1$, and the term
  $(\cR_{01}\mu)(\cR_{11}\mu)^n(\cR_{10}\mu)$ to $\tau=n+2$.
\end{proof}

Since $\cR\mu$ is linear, the algorithm above can be broken into to the
following steps
\begin{itemize}
\item Start with $\Eps_0=0$.
\item For $i=1,2,\ldots$ sequentially
 pick $P_n\sim \mu$, and pick $(\Eps_n,Q_n)$ with distribution
 \[
 \P(\Eps_n=j, Q_n=p |\Eps_{n-1}=i, Q_{n-1}) = (\cR_{ij} P_n)(p).
 \]
\item At the first time $\tau\ge 1$ so that $\Eps_\tau=0$, stop and return
  the pattern $Q_1\cdots Q_\tau$.
\end{itemize}

\section{Properties of $\pT$ and the sequence $\pT^k\mu$}
\label{s:general}

In this section we study the evolution of a measure $\mu$ on patterns under
iterated ascension.

\begin{proposition}\label{p:moments1}
  For a probability measure $\mu$ on patterns, the first moment of the pattern
  lengths with respect to $\pT\mu$ satisfies
  \[
  \pT\mu(|p|) \le m {\cA }\mu(|p|).
  \]
\end{proposition}
\begin{proof}
  Because length is sub-additive ($|pq|\leq |p|+|q|$), \eqref{eq:pT_def}
  implies that the first moment of the length of a sample from
  $\pT \mu$ is bounded above by the first moment of
  \[
  \theta = \widetilde{\cR}_{00} \mu + \widetilde{\cR}_{01} \mu
  \Big(1 - \widetilde{\cR}_{11} \mu \Big)^{-1} \widetilde{\cR}_{10}\mu,
  \]
  where $\widetilde{\cR}_{ij}$ refer to induced measures on length, and the
  above formula takes values in the semi-group algebra of $\Z_+$. Let
  $\varphi_{ij}(z) = \sum_n (\widetilde{\cR}_{ij} \mu)(n) z^n$ denote the
  generating functions for the measures $\widetilde{\cR}_{ij} \mu$. The
  generating function of the measure $\theta$ is
  \[
  f = \varphi_{00} + \varphi_{01} \big(1-\varphi_{11} \big)^{-1}
  \varphi_{10}.
  \]
  The first moment of  $\theta$ is given by $f'(1)$.

  We have formula \eqref{totalmass} for the total mass:
  \[
  \varphi_{ij}(1) = {\cR}_{ij} \mu (\cQ_k) = 1_{i=j}(1-\mu(D)) +
  \mu(D)\frac{(m-1)^j}m.
  \]
  It follows that
  \[
  f'(1) = \varphi_{00}'(1) + \varphi_{01}'(1) + (m-1) \big(
  \varphi_{10}'(1) + \varphi_{11}'(1) \big) = m \cA \mu(|p|). \qedhere
  \]
\end{proof}

\begin{lemma}\label{l:T mu facts}
  We have the following (where $D\subset\cQ_k$ is the set of patterns
  containing $q_{-1}$):
  \begin{enumerate}
  \item[(a)] For any $x\notin D$ we have $\pT \mu(x) \ge \mu(x)$.
  \item[(b)] If $\mu(D)>0$, then $\pT \mu(D) < \mu(D)$.
  \item[(c)] If $\mu$ is supported on $\cQ_{k-1} \cup \{q_k\}$ then so is
    $\pT \mu$, and $\mu(q_k)=\pT \mu(q_k)$.
  \end{enumerate}
\end{lemma}

\begin{proof}
  We refer to the algorithmic description of $\pT \mu$ and
  Lemma~\ref{L:cR_structure}. If the first pattern $P_1$ chosen from $\mu$
  is in $D^c$, then $\Eps_1=0$ and $Q_1=P_1$ is also the output of the
  procedure. Claim (a) follows.

  For part (b), assume that $\mu(p)>0$ for some $p\in D$. Consider $\cR p$
  as given by Lemma~\ref{L:cR_structure}. Taking only the $\frac1m
  \cR_{00}$ terms in the product shows that $\cR_{00} p$ assigns some
  positive probability $\delta \geq m^{-\ell}$ to the pattern $p$ with all
  appearances of $q_{-1}$ deleted. This pattern is in $D^c$. Thus $\pT
  \mu(D^c)\geq \mu(D^c) + \delta\mu(p)$.

  For part (c), note that for such measures, the only way to get $q_k$ in
  the ascension algorithm is if the first pattern selected from $\mu$ is
  $q_k$, in which case it is also the output.
\end{proof}

\begin{lemma}\label{L:limitsexist}
  If the measure sequence $\{\pT^k \mu\}$ is tight, then its limit exists
  and is supported on $D^c$.
\end{lemma}

\begin{proof}
  Since the measures are tight they have subsequential limits that are
  probability measures. Let $\nu$ be such a limit. $\pT $ is continuous, so
  there are $k$ so that $\nu$ and $\pT^k \mu$ are close enough to have
  \[
  \pT ^{k+1} \mu(D) - \pT \nu(D) < \eps.
  \]
  But again since $\nu$ is a limit point there is $\ell>k$ so that
  \[
  \nu(D) - \pT ^\ell \mu(D) < \eps.
  \]
  Summing the last two formulae and using that $\pT^\ell \mu(D) \leq
  \pT^{k+1} \mu(D)$ we get
  \[
  \nu(D) - \pT \nu(D) < 2\eps.
  \]
  Since $\eps$ is arbitrary, $\pT \nu(D)=\nu(D)$. By Lemma~\ref{l:T mu
    facts}(b), this implies that $\nu(D)=0$.

  Now, $\pT^k\mu(x)$ is monotone in $k$ for $x\notin D$, so all
  subsequential limits are equal.
\end{proof}

\begin{lemma}\label{L:bar_mu_ent_tight}
  Consider an entropy-tight set of measures $\{\mu_i\}$ on $\cQ_k$, such
  that $\mu_i(|p|) < M$ for some $M$. Assume further that for patterns
  $p\in \bigcup_i \operatorname{supp}(\mu_i) $ we have $H(\bar p) =
  o(|p|)$. Then $\{\bar{\mu_i}\}$ is also entropy-tight.
\end{lemma}

\begin{proof}
  Denote $K = \{g\in\M_{m,k} : |g| \leq \ell\}$, and consider a measure
  $\mu$ on $\cQ_k$. By definition,
  \[
  \bar{\mu}(\cdot) = \sum_p \mu(p) \bar{p}(\cdot).
  \]
  For compactness, denote $h(x) = -x\log x$, and note that $h$ is
  sub-additive. We have
  \[
  H_{K^c}(\bar{\mu}) = \sum_{|g|>\ell} h\big(\bar{\mu}(g)\big)
  \leq \sum_{|g|>\ell} \sum_p h\big(\mu(p) \bar{p}(g)\big).
  \]
  Since $\bar{p}$ is supported on group elements of length at most $|p|$,
  the above equals
  \[
  \sum_{|g|>\ell} \sum_{|p|>\ell} h\big(\mu(p) \bar{p}(g)\big)
  \leq \sum_{|p|>\ell} \sum_g h\big(\mu(p) \bar{p}(g)\big).
  \]
  However, since $\bar{p}$ is a probability measure, for any $a>0$ we have
  \[
  \sum_g h\big(a \bar{p}(g)\big) = -\sum_g (\log a +\log \bar{p}(g)) a
  \bar{p}(g) = h(a) + a H(\bar{p}),
  \]
  and therefore
  \begin{equation}\label{eq:res_ent_bound}
    H_{K^c}(\bar{\mu}) \leq \sum_{|p|>\ell} h(\mu(p)) + \sum_{|p|>\ell}
    \mu(p) H(\bar{p}).
  \end{equation}
  Now fix $\eps$. For $\ell$ sufficiently large, the first sum in
  \eqref{eq:res_ent_bound} is uniformly small for all $\mu$ in our
  entropy-tight family. Also by assumption, if $\ell$ is large enough
  then either
  $\mu(p)=0$ or $H(\bar{p}) < \eps|p|$. Thus for large enough $\ell$ we
  have
  \[
  H_{K^c}(\bar{\mu})
  \leq \eps + \sum_{|p|>\ell} \mu(p) \eps |p|
  \leq \eps + M \eps.
  \]
  Since $\eps$ is arbitrary, this implies that $\{\bar{\mu_i}\}$ is
  entropy-tight as claimed.
\end{proof}

\section{Preliminary results for entropy of patterns}
\label{s:bounded}

Here we establish some useful facts about entropy of $\bar{\mu}$ for some
measures $\mu$ on patterns. We first recover the
following result.

\begin{theorem}[\citet*{BKN}]\label{t:BKN}
  Let $\mu$ be supported on patterns of length $1$ in $\cQ_0$. Then
  $\Entlim(\bar{\mu})=0$, or equivalently, the entropy
  $\Ent(\bar{\mu}^{*n}) = o(n)$.
\end{theorem}

We present what is essentially the original proof, written
in the language of this paper.

\begin{proof}
  Lemma~\ref{l:T mu facts}(c) shows that $\pT^k\mu$ is supported on
  $\{q_0\}\cup\cQ_{-1} = \{q_0,q_{-1},\emptyset\}$. Thus the sequence
  $\pT^k\bar{\mu}$ is tight and entropy-tight. \lemref{limitsexist} shows
  that the $\nu = \lim \pT^k \mu$ exists and is supported on
  $\{q_0,\emptyset\}$. Thus $\supp(\bar{\nu})$ is contained in a finite
  group, and $\Entlim(\bar{\nu})=0$. The result now follows by
  Theorem~\ref{maintool}.
\end{proof}

Next, we show that evaluation of patterns on $\cQ_0$ have
entropy which is sub-linear in their length. We need the
following simple lemma.

\begin{lemma}[Monotonicity of pattern entropy]
  \label{l:monotonicity of pattern entropy}
  If $p$ and $r$ are patterns, then
  \[
  \Ent(\bar{pr}) \ge \max\left( \Ent(\bar{p}), \Ent(\bar{r})
  \right).
  \]
\end{lemma}

\begin{proof}
  $\bar{pr}$ is a convex combination of measures of the form
  $\bar{p}g$ where $g$ is chosen according to $\bar{Q}$. All these
  measures have the same entropy $\Ent(\bar{r})$. Entropy is a concave
  function, so by Jensen's inequality
  \[
  H(\bar{pr}) \geq H(\bar{p}).
  \]
  In the same way, $H(\bar{pr}) \geq H(\bar{r})$.
\end{proof}

\begin{lemma}\label{l:sublinear entropy of words in D}
  There is a function $f(n)=o(n)$ such that for any pattern $p\in \cQ_0$
  \[
  H(\bar{p}) \leq f(|p|).
  \]
\end{lemma}

\begin{proof}
  Let $\mu$ be the uniform measure on $\{q_{-1},q_0\}$, and
  $\mu_n=\mu^{*n}$. By concavity of entropy and Jensen's inequality, since
  $\bar{\mu_n}=\sum \mu_n(p) \bar{p}$, we have
  \[
  H(\bar{\mu_n}) \geq \sum_{p\in\cQ_0} \mu_n(p) \Ent(\bar{p})
  \geq \sum_{p\in A_\ell} \mu_n(p) \Ent(\bar{p})
  \]
  where $A_\ell$ consists of all patterns of length greater than $\ell$. By
  Lemma~\ref{l:monotonicity of pattern entropy}, the entropy of each word
  of length $n$ is greater than the entropy of each word of length $n-1$
  (since patterns in $\cQ_0$ are just alternating $q_0$'s
  and $q_{-1}$'s). It
  follows that
  \[
  H(\bar{\mu_n}) \geq \mu_n(A_\ell) \Ent(\bar{Q^*_\ell}),
  \]
  where $Q^*_\ell$ can be either of the two patterns of length $\ell$.

  The length of a sample of $\mu_n$ is the number of runs in the word $Q_1
  \cdots Q_n$ with i.i.d.\ letters. Since each letter starts a new run with
  probability $1/2$, the length is binomial, and symmetric about $n/2$.
  Thus with $n = 2\ell$ we have $\mu_n(A_\ell) \geq 1/2$. We find
  \[
  \Ent(Q^*_\ell) \leq 2 H(\bar{\mu_{2\ell}}) = o(\ell).     \qedhere
  \]
\end{proof}

We can also make some conclusions about linear automaton groups. It is not
hard to see that the combined supports $\L_0, \L_1$ of the evaluated
patterns $\bar q_0,\bar q_1$ generate a bounded automaton group. Certain
walks on this group also have zero asymptotic entropy.

\begin{lemma}\label{l:acentropy}
  Any measure $\mu$ on $\cQ_1$ supported on the patterns $\{q_0, q_1,
  \emptyset\}$ satisfies $\Entlim(\bar{\mu})=0$.
\end{lemma}

\begin{proof}
  Consider the subgroup of $\M_{m,1}$ consisting of automorphisms of the
  form
  \[
  g = \llangle g, g_1, \dots, g_{m-1} \rrangle  \rho,
  \]
  where $0.\rho = 0$.
  Note that the support of $\bar{\mu}$ is contained in this subgroup.
  Consider a random walk $X_k$ with step distribution $\bar{\mu}$, and
  write
  \[
  X_k= \llangle  X_k,Y_k(1), \dots, Y_k(m-1) \rrangle  \rho_k.
  \]
  The key observation is that the distribution of each $Y_1(i)$ is
  $\bar{\mu'}$ where $\mu' = \cR_{11}\mu$.
  It follows that for each $i>0$ the process $(Y_k(i),k\ge 0)$ is a random
  walk on $\M_{m,0}$ with step distribution $\bar{\mu'}$, where $\mu'$ is
  supported on $\{q_{-1},q_0,\emptyset\}$. By Theorem~\ref{t:BKN}, this
  random walk has zero asymptotic entropy. On the other hand, $X$ is
  determined by $\rho$ and the $Y_i$'s, so
  \[
  H(X_k) \leq H(\rho_k) +  \sum_{i=1}^{m-1} \Ent(Y_k(i))
  \leq \log m! + (m-1) \Ent(Y_k(1))
  \]
  Thus $\Entlim(\bar{\mu}) \leq (m-1) \Entlim(\bar{\mu'}) = 0$, as
  required.
\end{proof}

\section{The linear-activity case and the main theorem}
\label{s:linear}

The key step in the proof of the main theorem is the
following proposition. It is based on results from the
previous sections.

\begin{proposition}\label{p:tight and entropy-tight}
  Let $\mu$ be a measure on $\cQ_0\cup\{q_1\}$ with $\mu(|p|)<\infty$. 
  Then the sequence $\mu_k = \pT^k\mu$ is both tight and entropy-tight.
\end{proposition}

The tightness of the sequence depends on a certain contracting property of
the ascension operator $\pT$ for patterns on $\cQ_0$.
Proposition~\ref{p:moments1} gives bounds on moments of $\pT \mu$ in terms
of moments of $\cA\mu$. The next lemma bounds the moments of $\cA\mu$.

\begin{lemma}[Length of legacy]\label{L:R_moments}
  For a probability measure $\mu$ on $\{q_1\}\cup\cQ_0$, we have
  \[
  \cA \mu(|p|) \le 1 + \frac{m-1}{m^2} \mu(|p|)
  \]
\end{lemma}

\begin{proof}
  The probability measure $\mu$ is a convex combination of delta measures
  on a single patterns.
  Since $\cA\mu$ is linear, and the bounds are all affine in $\mu$, it
  suffices to prove the claim for delta measures $\mu = p$ supported on a
  single pattern $p\in\{q_1\}\cup\cQ_0$.

  If $p\in\{\emptyset,q_1,q_{-1}\}$, then $\cA\mu$ is supported on patterns
  of length at most $1$, and the claim is trivial.

  Otherwise, $\mu = p$ for some $p\in\cQ_0$. Let $t>0$ be the number of
  times $q_0$ appears in $p$, so that $|p|\geq 2t-1$. By \eqref{averaged
    legacy is a product}, a sample of $\cA\mu$ is given by $X_1 \cdots X_t$
  where $X_i$ are i.i.d.\ with distribution $\cA q_0 = \frac1m q_0 +
  \frac{m-1}{m} q_{-1}$. Each run of $q_0$'s or $q_{-1}$'s reduces to a
  single letter, so the length $R$ of $X_1 \cdots X_t$ is given by the
  number of runs.

  Let $N_i$ be the indicator of the event that $X_i$ starts a new run, so
  that $R = N_1 +\ldots + N_j$. Then $N_1\equiv 1$ and for $1<i\leq t$ we
  have $\E N_i = \frac{2(m-1)}{m^2}$. Now,
  \[
  \E R = 1 + \frac{2(m-1)}{m^2}(t-1) \le 1 + \frac{m-1}{m^2}|p|.
\qedhere
  \]

\end{proof}

Combining Lemma~\ref{L:R_moments} with Proposition~\ref{p:moments1} gives
the following contraction property of $\pT$:

\begin{corollary}[Contraction]\label{ascension and length Q1}
  For any probability measure $\mu$ on $\cQ_{0} \cup \{q_1\}$ we have
  \[
  \pT \mu (|p|) \le \frac{m-1}m \mu(|p|)+m
  \]
\end{corollary}

\begin{proof}[Proof of Proposition \ref{p:tight and entropy-tight}]
  Any non-negative sequence satisfying $x_{k+1}\le c+\alpha x_k$ with
  $\alpha<1$ must be bounded. By Corollary ~\ref{ascension and length Q1},
  the sequence $\mu_k(|p|)$ is bounded by some constant $C$, which implies
  tightness.


  For entropy-tightness, since $\mu_k$ is supported within
  $\{q_0,q_1,\emptyset\}\cup D$, it suffices to show that the contribution
  to its entropy from the set $D$ converges to $0$. Let $\nu_k$ denote
  $\mu_k$ conditioned to the set $D$. The contribution to $H(\mu_k)$ from
  $D$ is given by
  \begin{equation}\label{Dent}
    \sum_{P\in D} -\mu_k(P) \log \mu_k(P)
    = \eps_k H(\nu_k) - \eps_k \log \eps_k.
  \end{equation}
  We have
  \[
  \nu_k(|p|) \le \frac{\mu_k(|p|)}{\eps_k} \le \frac{C}{\eps_k}.
  \]
  However, note that $\nu_k$ is supported on patterns of alternating
  $q_0$'s and $q_{-1}$'s, so we have
  \begin{equation}\label{entropy and length entropy}
    H(\nu_k) \le 1 + H(\tilde \nu_k)
  \end{equation}
  where $\tilde \nu_k$ is the induced measure on lengths.
  Using Corollary ~\ref{l:entropy and expectation} we get that
  \[
  H(\tilde \nu_k) \leq 2 \log(C/\eps_k + 2)
  \]
  Combining this with \eqref{Dent}, \eqref{entropy and length entropy}, and
  \lemref{limitsexist} which shows that $\eps_k\to 0$, completes the proof.
\end{proof}

We are ready to prove the following, which implies Theorem~\ref{T:main}.

\begin{theorem}
  Let $\mu$ be a probability measure on patterns in $\cQ_0\cup \{q_1\}$
  with $\mu(|p|) < \infty$. The random walk on the linear-activity mother
  group with step distribution $\bar \mu$ has zero asymptotic entropy.
\end{theorem}

\begin{proof}
  By Theorem~\ref{maintool} it suffices to show that the
  sequence $T^k\bar{\mu}$ is entropy-tight and converges weakly to a
  measure with zero asymptotic entropy.

  Consider first the sequence $\pT^k\mu$ of measures on patterns.
  Proposition~\ref{p:tight and entropy-tight} shows this sequence is tight.
  Lemmas~\ref{l:T mu facts} and \ref{L:limitsexist} imply that $\nu=\lim
  \pT^k\mu$ exists and is supported on $\{q_1,q_0,\emptyset\}$. Finally, by
  Lemma \ref{l:acentropy}, evaluation of measures supported on
  $\{q_1,q_0,\emptyset\}$ have zero entropy, so $\Entlim(\bar{\nu})=0$.

  It remains to show that $T^k \bar{\mu}$ is entropy-tight. This follows
  from \lemref{bar_mu_ent_tight}, so we verify its conditions.
  Entropy-tightness and uniformly bounded expected length for $\pT^k \mu$
  are proved in Proposition~\ref{p:tight and entropy-tight}. The support of
  any $\pT^k\mu$ is contained in $\{q_1\}\cup\cQ_0$. By
  Lemma~\ref{l:sublinear entropy of words in D} the entropy of $p\in
  \{q_1\}\cup\cQ_0$ is $o(|p|)$.
\end{proof}

\paragraph{Acknowledgements.} We thank V. Nekrashevych for bringing the
example of the long-range group to our attention, as well as an
anonymous referee for pointing out a simplification of our proof.


\begin{thebibliography}{12}
\expandafter\ifx\csname
natexlab\endcsname\relax\def\natexlab#1{#1}\fi
\expandafter\ifx\csname url\endcsname\relax
  \def\url#1{{\tt #1}}\fi

\bibitem[Adian(1982)]{Ad82}
S. I. ~Adian(1982)
\newblock Random walks on free periodic groups.
\newblock {\em Izv. Akad. Nauk SSSR, Ser. Mat.}, 46(6)
(1982), 1139–1149.

\bibitem[Amir and Vir{\'a}g (2011)]{AV11}
G. ~Amir, B. ~Vir{\'a}g(2011)
\newblock A phase transition for the Liouville property of automaton groups.
\newblock {\em In preperation}

\bibitem[Bartholdi et~al.(2003)Bartholdi, Grigorchuk, and Nekrashevych]{BGN}
L.~Bartholdi, R.~Grigorchuk, and V.~V. Nekrashevych.
\newblock {\em From fractal groups to fractal sets}.
\newblock Fractals in {G}raz 2001.
\newblock Trends Math.
{25--118.} Birkh\"auser, Basel, 2003.

\bibitem[Bartholdi et~al.(2008)Bartholdi, Kaimanovich, and Nekrashevych]{BKN}
L.~Bartholdi, V.~A. Kaimanovich, and V.~V. Nekrashevych.
\newblock On amenability of automata groups, 2008.
\newblock \url{arXiv.org:0802.2837}. To appear in {\it Duke Math.
J.}

\bibitem[Bartholdi and Vir{\'a}g(2005)]{BV05}
L.~Bartholdi and B.~Vir{\'a}g (2005).
\newblock Amenability via random walks.
\newblock {\em Duke Math. J.}, 130\penalty0 (1):\penalty0 39--56.

\bibitem[Benjamini and Hoffman(2005)]{bh05}
I.~Benjamini and C.~Hoffman (2005).
\newblock {$\omega$}-periodic graphs.
\newblock {\em Electron. J. Combin.}, 12:\penalty0 Research Paper 46, 12 pp.
  (electronic).

\bibitem[Glasner and Mozes(2005)]{GlasnerMozes05}
Y.~Glasner and S.~Mozes (2005).
\newblock Automata and square complexes.
\newblock {\em Geom. Dedicata}, 111:\penalty0 43--64.

\bibitem[Grigorchuk and {\v{S}}uni{\'c}(2006)]{GS06}
R.~Grigorchuk and Z.~{\v{S}}uni{\'c} (2006).
\newblock Asymptotic aspects of {S}chreier graphs and {H}anoi {T}owers groups.
\newblock {\em C. R. Math. Acad. Sci. Paris}, 342\penalty0 (8):\penalty0
  545--550.

\bibitem[Grigorchuk(1984)]{Grigorchuk84}
R.~I. Grigorchuk (1984).
\newblock Degrees of growth of finitely generated groups and the theory of
  invariant means.
\newblock {\em Izv. Akad. Nauk SSSR Ser. Mat.}, 48\penalty0 (5):\penalty0
  939--985.

\bibitem[Grigorchuk and {\.Z}uk(2001)]{gz01}
R.~I. Grigorchuk and A.~{\.Z}uk (2001).
\newblock
The lamplighter group as a group generated by a 2-state
automaton, and its spectrum
\newblock {\em Geom. Dedicata}, 87\penalty0 (1-3):\penalty0
  209-244.

\bibitem[Grigorchuk and {\.Z}uk(2002)]{gz02a}
R.~I. Grigorchuk and A.~{\.Z}uk (2002).
\newblock On a torsion-free weakly branch group defined by a three state
  automaton.
\newblock {\em Internat. J. Algebra Comput.}, 12\penalty0 (1-2):\penalty0
  223--246.

\bibitem[Chatterji(2008)]{Guido}
{\it Guido's book of conjectures}
\newblock{(2008).}
\newblock Collected by I. Chatterji.
\newblock {\em Enseign. Math. (2)}, 54\penalty0 (1-2):\penalty0 3--189.

\bibitem[Kaimanovich(2005)]{Kaimanovich05}
V.~A. Kaimanovich (2005).
\newblock ``{M}\"unchhausen trick'' and amenability of self-similar groups.
\newblock {\em Internat. J. Algebra Comput.}, 15\penalty0 (5-6):\penalty0
  907--937.

\bibitem[Kemp(1997)]{Kemp}
A.~W. Kemp (1997).
\newblock Characterizations of a discrete normal distribution.
\newblock {\em J. Statist. Plann. Inference}, 63\penalty0 (2):\penalty0
  223--229.

\bibitem[Mazurov and Khukhro(2006)]{Kourovka}
\newblock {\em The {K}ourovka notebook}.
V.~D. Mazurov and E.~I. Khukhro, editors.
\newblock Russian Academy of Sciences, Novosibirsk, sixteenth edition, 2006.
\newblock Unsolved problems in group theory.

\bibitem[Milnor(1968)]{Milnor68}
J.~Milnor (1968).
\newblock Growth of finitely generated solvable groups.
\newblock {\em J. Differential Geometry}, 2:\penalty0 447--449.

\bibitem[Ol'shanski(1980)]{O80}
A.Y.~Olshanskii (1980).
\newblock , On the question of the existence of an invariant mean on a group (Russian).
\newblock {\em Uspekhi Mat. Nauk.} 35(4(214)) (1980), 199–200.

\bibitem[Ol'shanski and Sapir (2002)]{OS02}
A.Y.~Ol'shanskii and M.V. ~Sapir(2002).
\newblock Non-amenable finitely presented torsion-by-cyclic groups.
\newblock {\em Publ. Math. Inst. Hautes Études Sci.}  No. 96  (2002), 43--169 (2003).

\bibitem[Sidki(2000)]{Sidki00}
S.~Sidki (2000).
\newblock Automorphisms of one-rooted trees: growth, circuit structure, and
  acyclicity.
\newblock {\em J. Math. Sci. (New York)}, 100\penalty0 (1):\penalty0
  1925--1943.
\newblock Algebra, 12.

\bibitem[Sidki(2004)]{Sidki04}
S.~Sidki (2004).
\newblock Personal communication with the last author, at the conference ``Geometric group theory, random walks and harmonic
analysis'' Cortona, Italy, June 13-18, 2004

\bibitem[Sunic and Grigorchuk(2007)]{SG07}
  Z.~Sunic and R.~I.~Grigorchuk (2007).
  \newblock Self-similarity and branching in group theory.
  \newblock {\em London Mathematical Society} Lecture Note Series 339,
  36--95.

\bibitem[Vorobets and Vorobets(2007)]{Vorobets07}
M.~Vorobets and Y.~Vorobets (2007).
\newblock On a free group of transformations defined by an automaton.
\newblock {\em Geom. Dedicata}, 124:\penalty0 237--249.

\end{thebibliography}

\end{document}